\documentclass[12pt]{amsart}

\usepackage[margin=1in]{geometry}
\usepackage{amsmath,amssymb,amsthm,mathtools}
\usepackage[colorlinks=true,linkcolor=blue,citecolor=blue,urlcolor=blue]{hyperref}
\usepackage[capitalise,nameinlink]{cleveref}

\newtheorem{theorem}{Theorem}[section]
\newtheorem{proposition}[theorem]{Proposition}
\newtheorem{lemma}[theorem]{Lemma}
\newtheorem{corollary}[theorem]{Corollary}

\theoremstyle{definition}

\newtheorem{remark}[theorem]{Remark}

\newcommand{\PP}{\mathbb{P}}
\newcommand{\Aff}{\mathbb{A}}
\newcommand{\NN}{\mathbb{N}}

\newcommand{\mm}{\mathfrak m}
\newcommand{\ord}{\operatorname{ord}}
\newcommand{\Supp}{\operatorname{Supp}}

\title[Demailly's inequality in positive characteristic]
{Demailly's Inequality for Finite Sets of Points\\in Positive Characteristic}

\author{Piotr Pokora and Tomasz Szemberg}
\date{\today}

\begin{document}

\begin{abstract}
Let $k$ be an algebraically closed field of characteristic $p>0$, $n\ge1$, and let
$X\subset \PP^n_k$ be a finite nonempty set of distinct points with the defining
ideal $I=I(X)$. We give a proof of 
Demailly's inequality in positive characteristic
$$
 \widehat\alpha(I)\ge
 \frac{\alpha(I^{(m)})+n-1}{m+n-1}
 \qquad (m\ge 1).
$$
The argument is based on the Frobenius--Hasse derivative method used in this context by
H\`a and Sivakumar~\cite{HaSivakumar2026}. The key additional observation is a strict-growth
lemma for the initial degrees of symbolic powers of a finite set of affine
points:
$$
 \alpha(J^{(t)})\ge \alpha(J^{(t-1)})+1\qquad(t\ge1).
$$
In characteristic $p$, if a minimum-degree polynomial has a nonzero first
ordinary derivative, this follows by differentiation; if all first ordinary
derivatives vanish, perfectness gives a $p$th root and an induction on the
symbolic exponent. This removes the lower bound on $p$ appearing in
\cite[Lemma~3.7 and Theorem~3.8]{HaSivakumar2026}. The remainder of the proof
uses the $q=p^e$ Frobenius decomposition and a maximal Hasse derivative.
%as in \cite[Section~3]{HaSivakumar2026}.
\end{abstract}

\maketitle

\section{Introduction}

For a nonzero homogeneous ideal $I\subseteq S$, we denote its initial degree by
\[
\alpha(I)=\min\{d\mid I_d\neq 0\}.
\]
In this work we are interested in homogeneous ideals
$I=I(X)\subset S=k[x_0,\ldots,x_n]$ defining a finite set
$X\subset \mathbb{P}^n_k$ of distinct $k$-rational points. For ideals of
points the sequence $r\mapsto \alpha(I^{(r)})$ is subadditive, so the limit
\[
\widehat{\alpha}(I)=\lim_{r\to\infty}\frac{\alpha(I^{(r)})}{r}
\]
is well defined by Fekete's lemma; see also~\cite[Lemma~2.3.1]{BocciHarbourne2010}.
The number $\widehat{\alpha}(I)$ is the Waldschmidt constant of $I$.

Symbolic powers of point ideals may be identified with prescribed orders of
vanishing; this is a standard instance of the Zariski--Nagata theorem
\cite{EisenbudHochster1979,Nagata1962,Zariski1949}.

Demailly conjectured that for every $m\geq 1$ one has
\begin{equation}\label{eq:Demailly}
\widehat{\alpha}(I)\geq
\frac{\alpha(I^{(m)})+n-1}{m+n-1}.
\end{equation}
For $m=1$ this specializes to Chudnovsky's inequality. The original
conjectures arose over the complex numbers in the works of Chudnovsky
\cite{Chudnovsky1981} and Demailly~\cite{Demailly1982}.

Recently H\`a and Sivakumar~\cite{HaSivakumar2026} proved \eqref{eq:Demailly} in
characteristic zero by first proving a finite Frobenius estimate after
reduction to sufficiently large positive characteristic and then specializing
back to characteristic zero. Their positive-characteristic argument contains
the following estimate; see~\cite[Theorem~3.8]{HaSivakumar2026}. If $J$ is the ideal of
finitely many affine points, $q=p^e$, and
\[
M=q(m+n-1)-(n-1),
\]
then, under a lower bound on $p$, one has
\begin{equation}\label{eq:frob-estimate-intro}
\alpha(J^{(M)})\geq q\alpha(J^{(m)})+(q-1)(n-1).
\end{equation}

The point of the present paper is that this large-characteristic restriction
is not intrinsic to the Frobenius--Hasse derivative method. It enters only at
the step where one wants to differentiate a polynomial of minimum degree in a
symbolic power. In small characteristic such a polynomial may be a
nonconstant $p$th power, in which case all first ordinary derivatives vanish.
Our key observation is that this exceptional case still forces the required
growth: over a perfect field one can take a $p$th root, which lowers the
symbolic exponent from $t$ to $\lceil t/p\rceil$, and then argue inductively.
Thus the Frobenius phenomenon responsible for the characteristic restriction
can be absorbed recursively rather than excluded by a numerical assumption
on $p$.

More precisely, we prove that if $J$ is the radical ideal of a nonempty finite
set of points in affine space over a perfect field of characteristic $p>0$,
then
\[
\alpha(J^{(t)})\geq \alpha(J^{(t-1)})+1
\qquad (t\geq 1).
\]
This strict-growth statement holds in every positive characteristic. It
replaces the large-characteristic growth lemma used in~\cite{HaSivakumar2026} and removes
the lower bound on $p$ from \eqref{eq:frob-estimate-intro}. The remaining
part of the argument uses the same $q=p^e$ Frobenius decomposition and a
maximal Hasse derivative as in~\cite{HaSivakumar2026}.

Our main result is the following.

\begin{theorem}\label{thm:main}
Let $k$ be an algebraically closed field of characteristic $p>0$ and let
$n\geq 1$. Let $X\subset \mathbb{P}^n_k$ be a finite nonempty set of distinct
points, with defining ideal $I=I(X)$. Then for every integer $m\geq 1$,
\[
\widehat{\alpha}(I)\geq
\frac{\alpha(I^{(m)})+n-1}{m+n-1}.
\]
\end{theorem}

\section{Affine points, symbolic powers, and dehomogenization}

We first pass to the affine version of the interpolation problem. To this end, let
$
R=k[x_1,\dots,x_n]
$
and let $P_1,\dots,P_s\in\Aff^n_k$ be distinct points. Let
$
\mm_i=(x_1-a_{i1},\dots,x_n-a_{in})
$
be the maximal ideal of 
$P_i=(a_{i1},\ldots,a_{in})$, and let
$$
J=\bigcap_{i=1}^s\mm_i
$$
be the ideal of the union of points $X$.

Since the $\mm_i$ are distinct maximal ideals, the symbolic powers of $J$ are
\begin{equation}\label{eq:symbolic-affine}
J^{(r)}=\bigcap_{i=1}^s\mm_i^r \qquad(r\ge1).
\end{equation}
For any nonzero ideal $L\subseteq R$, not necessarily homogeneous, define its initial degree
$$
\alpha(L)=\min\{\deg f\mid 0\neq f\in L\},
$$
where $\deg$ denotes the total degree.
More generally, we consider the initial sequence of $L$,
$a_t:=\alpha(L^{(t)})$ for $t\ge1$ with the convention $a_0=0$ (or equivalently, $J^{(0)}=R$). As in the homogeneous case, we set
$\widehat\alpha(L)=\lim_{t\to\infty}a_t/t$.

For a point $P\in\Aff^n_k$ with maximal ideal $\mm_P$ and $0\neq f\in R$, define the order of vanishing at $P$ as
$$
\ord_P(f)=\max\{r\ge0\mid f\in\mm_P^r\}.
$$
Thus
$$
 f\in J^{(r)}
 \quad\Longleftrightarrow\quad
 \ord_{P_i}(f)\ge r\mbox{ for every }i.
$$

We shall use the following projective-to-affine reduction. It is the same
statement as \cite[Theorem~2.1]{HaSivakumar2026}; we include the argument for the sake of
completeness.

\begin{proposition}\label{prop:dehom}
Let $X=\{Q_1,\dots,Q_s\}\subset\PP^n_k$ be a finite set of points and $k$ be infinite. After a
linear change of projective coordinates we can assume that every $Q_i$ lies in the affine
chart $x_0\neq0$. So we have
$Q_i=[1:a_{i1}:\cdots:a_{in}]$.
Let
$$
J=\bigcap_{i=1}^s (y_1-a_{i1},\dots,y_n-a_{in})
\subset k[y_1,\dots,y_n].
$$
Then for every $r\ge1$,
$\alpha(I(X)^{(r)})=\alpha(J^{(r)})$ 
and consequently,
$\widehat\alpha(I(X))=\widehat\alpha(J)$.
\end{proposition}

\begin{proof}
Since $k$ is infinite, there exists a hyperplane avoiding
all points of $X$. After a linear change of coordinates, we may assume that this hyperplane
is $x_0=0$.

For any $d\ge0$ dehomogenization gives a vector-space isomorphism
$$
\phi_d:S_d\ni F\longmapsto F(1,y_1,\dots,y_n) \in R_{\le d},
$$
whose inverse is degree-$d$ homogenization
$$
 \psi_d(f)=x_0^d
 f\!\left(\frac{x_1}{x_0},\dots,\frac{x_n}{x_0}\right).
$$
Setting
$z_{ij}=x_j-a_{ij}x_0$ and
$w_{ij}=y_j-a_{ij}$ we have
$I(Q_i)=(z_{i1},\dots,z_{in})$ and
$\phi_1(z_{ij})=w_{ij}$.

For arbitrary $F\in I(Q_i)^r\cap S_d$, every term in the expression of $F$ in the coordinates
$z_{ij}$ has total $z_i$-degree at least $r$, so dehomogenization gives
$\phi_d(F)\in\mm_i^r$.
Thus
$$
 F\in I(X)^{(r)}=\bigcap_i I(Q_i)^r
 \quad\Longrightarrow\quad
 \phi_d(F)\in\bigcap_i\mm_i^r=J^{(r)}.
$$
Since dehomogenization does not increase the total degree,
$$
 \alpha(I(X)^{(r)})\ge\alpha(J^{(r)}).
$$

Conversely, let $0\neq f\in J^{(r)}$ be a polynomial of degree $d$. For each $i$, we expand
$f$ in the translated coordinates $w_{i1},\dots,w_{in}$:
$$
 f=\sum_{\beta\in\NN^n}c_{i,\beta}w_i^\beta.
$$
Since $f\in\mm_i^r$, we have $c_{i,\beta}=0$ whenever $|\beta|<r$. 
Moreover, $|\beta|\le d$ for every term that occurs. Hence
$$
 \psi_d(f)
 =\sum_{|\beta|\ge r}
 c_{i,\beta}x_0^{d-|\beta|}z_i^\beta
 \in I(Q_i)^r.
$$
This holds for every $i$, so
$\psi_d(f)\in I(X)^{(r)}$.
As $\psi_d(f)$ is a nonzero homogeneous form of degree $d$, we obtain
$$
 \alpha(I(X)^{(r)})\le d.
$$
Taking $d=\alpha(J^{(r)})$ proves equality of the initial degrees. Equality of the Waldschmidt
constants follows by dividing by $r$ and taking the limit.
\end{proof}

Hence it is enough to prove \cref{thm:main} for finite sets of affine points.

\section{Differentiation and Frobenius}

For this section, it is enough to assume that $k$ is a perfect field of characteristic $p>0$. Recall that in characteristic $p$, a field is perfect if and only if its Frobenius map $a\to a^p$ is surjective; see \cite[Chapter V, Corollary 6.12]{LangAlgebra}. In particular, algebraically closed fields are perfect.

\subsection{Ordinary derivatives and $p$th powers}

\begin{lemma}\label{lem:pth-power}
Let $f\in k[x_1,\ldots,x_n]$. If
$$
 \frac{\partial f}{\partial x_i}=0
 \mbox{ for all }1\le i\le n,
$$
then
$f\in k[x_1^p,\dots,x_n^p]$.
If $k$ is perfect, then there exists $g\in R$ such that
$f=g^p$.
\end{lemma}

\begin{proof}
Let
$$
 f=\sum_{\beta\in\NN^n}c_\beta x^\beta,
 \mbox{ where }
 x^\beta=x_1^{\beta_1}\cdots x_n^{\beta_n}.
$$
Then
$$
 \frac{\partial f}{\partial x_i}
 =\sum_{\beta_i>0}
 c_\beta\beta_i
 x_1^{\beta_1}\cdots x_i^{\beta_i-1}\cdots x_n^{\beta_n}.
$$
For fixed $i$, distinct exponent vectors $\beta$ give distinct monomials
$x^{\beta-e_i}$, so there is no cancellation among these displayed terms.
Therefore, if $\partial f/\partial x_i=0$, then for every $\beta$ with
$c_\beta\neq0$ we get
$\beta_i=0\quad\text{in }k$,
and hence $p\mid\beta_i$. If all first partial derivatives vanish, every
coordinate of every exponent vector in the support of $f$ is divisible by
$p$. Thus
$$
 f=\sum_\mu c_{p\mu}x^{p\mu}.
$$
Since $k$ is perfect, there are $b_\mu\in k$ with
$b_\mu^p=c_{p\mu}$. The Frobenius map is a ring
homomorphism, so
$$
 f
 =\sum_\mu b_\mu^p x^{p\mu}
 =\left(\sum_\mu b_\mu x^\mu\right)^p
$$
and we are done with $g=\sum_\mu b_\mu x^\mu$.
\end{proof}

\begin{lemma}\label{lem:order-frob}
Let $P\in\Aff^n_k$, $0\neq g\in R$, and let $q=p^e$. Then
$$
 \ord_P(g^q)=q\,\ord_P(g).
$$
Consequently, for every integer $s\ge1$,
$$
 g^q\in\mm_P^s
 \quad\Longleftrightarrow\quad
 g\in\mm_P^{\lceil s/q\rceil}.
$$
\end{lemma}

\begin{proof}
We translate coordinates so that $P$ is the origin. Let
$g=\sum_\beta c_\beta x^\beta$
and
$d=\min\{|\beta|\mid c_\beta\neq0\}=\ord_P(g)$.
Since $q$ is a power of $p$,
$$
 g^q=\sum_\beta c_\beta^q x^{q\beta}.
$$
Its least total degree is $qd$, so
$\ord_P(g^q)=qd$. The final equivalence is immediate.
\end{proof}

\begin{lemma}
\label{lem:ordinary-order}
For every $f\in R$, every point $P$, and every $i$, there is
$$
 \frac{\partial f}{\partial x_i}
 \in\mm_P^{\max\{\ord_P(f)-1,0\}}.
$$
In particular, if $f\in J^{(t)}$, then
$$
 \frac{\partial f}{\partial x_i}\in J^{(t-1)}
 \qquad(t\ge1).
$$
\end{lemma}

\begin{proof}
After translating $P$ to the origin, every monomial of a polynomial in
$\mm_P^t$ has total degree at least $t$. Differentiating a monomial either
kills it or decreases its total degree by exactly one. Hence
$\partial(\mm_P^t)/\partial x_i\subseteq\mm_P^{t-1}$. We apply this at each
point $P_j$ and use \eqref{eq:symbolic-affine}.
\end{proof}

\subsection{Hasse derivatives and the $q$-Frobenius decomposition}

For the large Frobenius estimate we use Hasse derivatives, following the lines of
\cite[Section~3]{HaSivakumar2026}. If $\lambda,\beta\in\NN^n$, we define
$$
 D_\lambda(x^\beta)=
 \begin{cases}
 \displaystyle\binom{\beta}{\lambda}x^{\beta-\lambda},
 &\beta\ge\lambda\text{ coordinatewise},\\[5pt]
 0,&\text{otherwise},
 \end{cases}
$$
and extend $k$-linearly.
We use the convention
$$
 \binom{\beta}{\lambda}
 =\prod_{i=1}^n\binom{\beta_i}{\lambda_i}.
$$
Equivalently, $D_\lambda(f)$ is the coefficient of $T^\lambda$ in the
Hasse--Taylor expansion
\begin{equation}\label{eq:HasseTaylor}
 f(x+T)=\sum_{\lambda\in\NN^n}D_\lambda(f)(x)T^\lambda.
\end{equation}
The following lemma generalizes \cref{lem:ordinary-order}.
\begin{lemma}\label{lem:Hasse-order}
For every point $P$ and every $s\ge0$, one has
$$
 D_\lambda(\mm_P^s)
 \subseteq
 \mm_P^{\max\{s-|\lambda|,0\}}.
$$
Consequently,
$$
 D_\lambda(J^{(s)})
 \subseteq
 J^{(\max\{s-|\lambda|,0\})}.
$$
\end{lemma}

\begin{proof}
We translate $P$ to the origin. If $|\beta|\ge s$, then
$D_\lambda(x^\beta)$ is either zero or a scalar multiple of
$x^{\beta-\lambda}$, whose degree is
$|\beta|-|\lambda|\ge s-|\lambda|$. This proves the first inclusion. The
second follows by intersecting the corresponding inclusions for the maximal ideals of each point.
\end{proof}

\begin{lemma}\label{lem:Hasse-qpower}
Let $q=p^e$. For $h\in R$ we have
$D_\lambda(h^q)=0$ unless $\lambda=q\mu$ for some $\mu\in\NN^n$ and
$D_{q\mu}(h^q)=D_\mu(h)^q$.
\end{lemma}

\begin{proof}
We apply \eqref{eq:HasseTaylor} to $h$ and use the fact that the $q$th-power Frobenius is
a ring homomorphism:
$$
 h(x+T)^q
 =\left(\sum_\mu D_\mu(h)(x)T^\mu\right)^q
 =\sum_\mu D_\mu(h)(x)^qT^{q\mu}.
$$
On the other hand,
$$
 h(x+T)^q
 =\sum_\lambda D_\lambda(h^q)(x)T^\lambda.
$$
Comparing coefficients of $T^\lambda$ provides the assertion.
\end{proof}

\begin{lemma}\label{lem:frob-decomp}
Let $q=p^e$ and
$\Lambda_q=\{0,1,\dots,q-1\}^n$.
Every polynomial $f\in R$ can be written uniquely as
\begin{equation}\label{eq:frob-decomp}
 f=\sum_{\lambda\in\Lambda_q}h_\lambda^q x^\lambda
 \mbox{ for some } h_\lambda\in R.
\end{equation}
Moreover, the monomial supports of the nonzero summands
$h_\lambda^q x^\lambda$ lie in pairwise distinct residue classes modulo $q$
in $\NN^n$.
\end{lemma}

\begin{proof}
Every exponent vector $\beta\in\NN^n$ 
can be uniquely decomposed as
$\beta=q\mu+\lambda$ with $\mu\in\NN^n$ and $\lambda\in\Lambda_q$.

Since $k$ is perfect, every scalar $c\in k$ has a unique $q$th root in
$k$. Grouping the monomials of $f$ according
to the residue class $\lambda$ of their exponent vector modulo $q$ gives
\eqref{eq:frob-decomp}. Uniqueness follows because two distinct residue
classes modulo $q$ have disjoint monomial supports.
\end{proof}

\begin{lemma}
\label{lem:isolate}
Let $f$ be as in \eqref{eq:frob-decomp}. We set 
$\Supp_q(f)=\{\lambda\in\Lambda_q\mid h_\lambda\neq0\}$.
Let $\gamma$ be maximal in $\Supp_q(f)$ for the coordinatewise partial order
on $\NN^n$. Then
$$
 D_\gamma(f)=h_\gamma^q.
$$
\end{lemma}

\begin{proof}
Fix $\lambda\in\Supp_q(f)$. By the Hasse product formula we have
$$
D_\gamma(h_\lambda^qx^\lambda)
=\sum_{\alpha+\beta=\gamma}
D_\alpha(h_\lambda^q)D_\beta(x^\lambda).
$$
By \cref{lem:Hasse-qpower}, the first factor can be nonzero only when
$\alpha\in q\NN^n$. But $0\le\alpha\le\gamma$ coordinatewise and
$\gamma_i<q$ for every $i$, so it must be $\alpha=0$ and consequently $\beta=\gamma$. Thus
$$
 D_\gamma(h_\lambda^qx^\lambda)
 =h_\lambda^qD_\gamma(x^\lambda).
$$
This is zero unless $\lambda\ge\gamma$ coordinatewise. By maximality of
$\gamma$ in $\Supp_q(f)$, the only nonzero summand for which this can occur is
$\lambda=\gamma$. Note that
$D_\gamma(x^\gamma)=1$. Summing over $\lambda$ gives
$D_\gamma(f)=h_\gamma^q$.
\end{proof}

\section{Strict growth of symbolic initial degrees}

In this section we will show the crucial fact that the initial sequence of an ideal of affine points is strictly increasing.
\begin{proposition}\label{prop:strict-growth}
Let $k$ be a perfect field of characteristic $p>0$, let $n\ge1$, and let $J$ be the radical
ideal of a nonempty finite set of $k$-rational points in $\Aff^n_k$. Then for
every integer $t\ge1$,
\begin{equation}\label{eq:strict-growth}
 a_t\ge a_{t-1}+1.
\end{equation}
More generally, for all integers $t\ge m\ge0$,
\begin{equation}\label{eq:linear-growth}
 a_t\ge a_m+t-m.
\end{equation}
\end{proposition}

\begin{proof}
We prove \eqref{eq:strict-growth} by induction on $t$.

For $t=1$, the ideal $J$ is proper, so it contains no nonzero constant.
Therefore $a_1\ge1=a_0+1$.

Now we assume that $t\ge2$ and that \eqref{eq:strict-growth} is known for every positive
integer strictly smaller than $t$. Let 
$f\in J^{(t)}$ be a nonzero polynomial with
$\deg f=a_t$.

There are two cases.

\smallskip
\textbf{Case 1.} There exists $1\leq i\leq n$ such that the first ordinary derivative $\partial f/\partial x_i$ of $f$ is nonzero. Then by \cref{lem:ordinary-order},
$\partial f/\partial x_i\in J^{(t-1)}$ and it must be
$$
a_{t-1}
\le\deg(\partial f/\partial x_i)
\le a_t-1.
$$
Hence $a_t\ge a_{t-1}+1$ in this case.

\smallskip
\textbf{Case 2.} All first ordinary derivatives of $f$ vanish.
Then by \cref{lem:pth-power}, there exists $g\in R$ such that $f=g^p$.
By \cref{lem:order-frob}, from 
$f=g^p\in J^{(t)}$ we obtain
$g\in J^{(u)}$,
with
$u=\left\lceil\frac{t}{p}\right\rceil$.
Therefore
\begin{equation}\label{eq:at-lower-pau}
 a_t=\deg f=p\deg g\ge p a_u.
\end{equation}
Conversely, if $G\in J^{(u)}$ is a nonzero polynomial of degree $a_u$, then
$G^p\in J^{(pu)}\subseteq J^{(t)}$,
because $pu\ge t$. Hence
$a_t\le p a_u$, which combined with \eqref{eq:at-lower-pau} gives
$a_t=p a_u$ and $u=\left\lceil t/p\right\rceil$.

There is a unique $1\le r\le p$ such that 
$t=p(u-1)+r$.
Since $t\ge2$, we have $u<t$. Thus the induction hypothesis applies at $u$ and yields
\begin{equation}\label{eq:induction-u}
 a_{u-1}\le a_u-1.
\end{equation}
We choose two nonzero polynomials:
$G\in J^{(u)}$ with $\deg G=a_u$
and, if $u\ge2$,
$H\in J^{(u-1)}$ with $\deg H=a_{u-1}$.
If $u=1$, we take $H=1$. Note that this is consistent with $J^{(0)}=R$ and $a_0=0$.

Let $F=G^{r-1}H^{p-r+1}$. Since $R$ is a domain, $F\neq0$. Then
(using $J^{(a)}J^{(b)}\subseteq J^{(a+b)}$), 
we have
$F\in J^{((r-1)u+(p-r+1)(u-1))}$.
The exponent can be simplified as follows:
\begin{align*}
 (r-1)u+(p-r+1)(u-1)
 &=pu-(p-r+1)\\
 &=p(u-1)+r-1\\
 &=t-1.
\end{align*}
Thus
$F\in J^{(t-1)}$ 
and consequently $a_{t-1}\le\deg F$.

On the other hand, by \eqref{eq:induction-u},
\begin{align*}
 \deg F
 &=(r-1)a_u+(p-r+1)a_{u-1}\\
 &\le (r-1)a_u+(p-r+1)(a_u-1)\\
 &=p a_u-(p-r+1)\\
 &\le p a_u-1\\
 &=a_t-1,
\end{align*}
so that $\deg F\le a_t-1$.

Therefore $a_t\ge a_{t-1}+1$ also in Case~2. This completes the induction.

The claim in \eqref{eq:linear-growth} follows immediately.
\end{proof}
At this point it is worth comparing our approach with that of H\`a--Sivakumar.
\begin{remark}\label{rem:HS-comparison}
Under the additional hypothesis
$p>\alpha(J^{(m+n-1)})$, H\`a and Sivakumar prove
$$
 \alpha(J^{(t)})\ge\alpha(J^{(m)})+t-m,
 \mbox{ with } m\le t\le m+n-1,
$$
by differentiating a minimum-degree polynomial; see
\cite[Lemma~3.7]{HaSivakumar2026}. Their lower bound on $p$ ensures that such
a polynomial has degree $<p$, so the case in which all first derivatives
vanish cannot occur for a nonconstant polynomial. Proposition
\ref{prop:strict-growth} handles exactly that omitted Frobenius case: when all
first derivatives vanish, one takes a $p$th root and reduces the symbolic
exponent from $t$ to $\lceil t/p\rceil<t$.
\end{remark}

\section{The Frobenius estimate in arbitrary positive characteristic}

We now rerun the Frobenius--Hasse argument of
\cite[Theorem~3.8]{HaSivakumar2026}, replacing their large-characteristic
growth lemma by \cref{prop:strict-growth}.

\begin{theorem}\label{thm:frob-estimate}
Let $k$ be a perfect field of characteristic $p>0$, let $n\ge1$, and let $J$ be the radical
ideal of a nonempty finite set of $k$-rational points in $\Aff^n_k$. We fix $m\ge1$ and let
$q=p^e$ be a power of $p$ with some $e\ge1$. Then for
$M=q(m+n-1)-(n-1)$ there is
\begin{equation}\label{eq:uniform-frob}
 \alpha(J^{(M)})
 \ge q\alpha(J^{(m)})+(q-1)(n-1).
\end{equation}
\end{theorem}

\begin{proof}
Let $f\in J^{(M)}$ be a nonzero polynomial of minimal total degree 
$\deg f=\alpha(J^{(M)})$.

By \cref{lem:frob-decomp}, $f$ can be written uniquely in the form
\begin{equation}\label{eq:f-decomp-proof}
f=\sum_{\lambda\in\Lambda_q}h_\lambda^q x^\lambda.
\end{equation}
Let $\gamma\in\Supp_q(f)$ be a coordinatewise maximal element in $\Supp_q(f)$.

By \cref{lem:isolate}, 
$D_\gamma(f)=(h_\gamma)^q$.
Since $f\in J^{(M)}$, \cref{lem:Hasse-order} implies
$$
 (h_\gamma)^q=D_\gamma(f)
 \in J^{(M-|\gamma|)}.
$$
By \cref{lem:order-frob},
$h_\gamma\in\mm_i^t$
for every $i$, where
\begin{equation}\label{eq:t-definition}
t=\left\lceil\frac{M-|\gamma|}{q}\right\rceil.
\end{equation}
Thus $h_\gamma\in J^{(t)}$.

We next estimate $t$. Since $\gamma\in\Lambda_q$,
$|\gamma|\le n(q-1)$. Using the definition of $M$ we obtain
\begin{align*}
 M-|\gamma|
 &\ge q(m+n-1)-(n-1)-n(q-1)\\
 &=q(m-1)+1.
\end{align*}
Therefore
$$
 t\ge
 \left\lceil m-1+\frac1q\right\rceil
 =m.
$$
On the other hand, $|\gamma|\ge0$, so
$$
 t\le
 \left\lceil\frac{M}{q}\right\rceil
 =\left\lceil
 m+n-1-\frac{n-1}{q}
 \right\rceil
 \le m+n-1.
$$
Thus
\begin{equation}\label{eq:t-range}
 m\le t\le m+n-1.
\end{equation}

Since $h_\gamma\in J^{(t)}$, inequality $\eqref{eq:t-range}$ allows us to apply \eqref{eq:linear-growth} and we get
\begin{equation}\label{eq:degree-hgamma}
 \deg h_\gamma
 \ge\alpha(J^{(t)})
 \ge\alpha(J^{(m)})+t-m.
\end{equation}

The summands in \eqref{eq:f-decomp-proof} have pairwise disjoint monomial
supports, because their exponent vectors lie in distinct residue classes
modulo $q$. Hence no monomial cancellation can occur between two different
summands, and in particular
\begin{equation}\label{eq:degree-f-component}
 \deg f
 \ge\deg(h_\gamma^qx^\gamma)
 =q\deg h_\gamma+|\gamma|.
\end{equation}
Combining \eqref{eq:degree-hgamma} and
\eqref{eq:degree-f-component},
\begin{align*}
 \deg f
 &\ge q\alpha(J^{(m)})+q(t-m)+|\gamma|.
\end{align*}
By the definition \eqref{eq:t-definition} of $t$,
$qt\ge M-|\gamma|$ and thus
\begin{align*}
 q(t-m)+|\gamma|
 &\ge M-qm\\
 &=q(m+n-1)-(n-1)-qm\\
 &=(q-1)(n-1).
\end{align*}
We conclude that
$$
 \deg f
 \ge q\alpha(J^{(m)})+(q-1)(n-1).
$$
Since $\deg f=\alpha(J^{(M)})$, this is exactly
\eqref{eq:uniform-frob}.
\end{proof}

\begin{remark}
The only change from the proof of \cite[Theorem~3.8]{HaSivakumar2026} is the
replacement of \cite[Lemma~3.7]{HaSivakumar2026} by
\cref{prop:strict-growth}. In particular, the Frobenius decomposition, the
maximal-support Hasse derivative, the range
$m\le t\le m+n-1$, and the final degree computation are the same mechanisms
as in that preprint.
\end{remark}

\section{Demailly's inequality in characteristic $p$}

\begin{corollary}[Affine Demailly inequality]\label{cor:affine-demailly}
Let $k$ be a perfect field of characteristic $p>0$, let $n\ge1$, and let $J$ be the radical
ideal of a nonempty finite set of $k$-rational points in $\Aff^n_k$. Then for
every $m\ge1$,
$$
 \widehat\alpha(J)
 \ge
 \frac{\alpha(J^{(m)})+n-1}{m+n-1}.
$$
\end{corollary}

\begin{proof}
For $e\ge1$, put $q_e=p^e$ and let 
$M_e=q_e(m+n-1)-(n-1)$.
By \cref{thm:frob-estimate},
\begin{equation}\label{eq:ratio-e}
 \frac{\alpha(J^{(M_e)})}{M_e}
 \ge
 \frac{
 q_e\alpha(J^{(m)})+(q_e-1)(n-1)
 }{
 q_e(m+n-1)-(n-1)
 }.
\end{equation}
As $e\to\infty$, one has $M_e\to\infty$. Since the sequence
$\alpha(J^{(r)})/r$
converges to $\widehat\alpha(J)$, the subsequence indexed by $M_e$ has the same limit. Therefore the left-hand side of \eqref{eq:ratio-e}
converges to $\widehat\alpha(J)$. The right-hand side converges to
$$
 \frac{\alpha(J^{(m)})+n-1}{m+n-1}.
$$
So passing to the limit proves the assertion.
\end{proof}
We conclude by proving our main result.
\begin{proof}[Proof of \cref{thm:main}]
Let $X\subset\PP^n_k$ be as in \cref{thm:main}. Since $k$ is algebraically
closed, it is infinite and perfect. Choosing a hyperplane avoiding $X$ and using
\cref{prop:dehom} we obtain an affine point ideal $J\subset k[y_1,\dots,y_n]$
such that, for every $r\ge1$,
$$
 \alpha(I(X)^{(r)})=\alpha(J^{(r)}),
$$
and therefore
$$
 \widehat\alpha(I(X))=\widehat\alpha(J).
$$
Applying \cref{cor:affine-demailly} to $J$:
$$
 \widehat\alpha(I(X))
 =\widehat\alpha(J)
 \ge
 \frac{\alpha(J^{(m)})+n-1}{m+n-1}
 =
 \frac{\alpha(I(X)^{(m)})+n-1}{m+n-1}
$$
proves \cref{thm:main}.
\end{proof}
\section*{Funding}
Piotr Pokora is supported by the National Science Centre (Poland) Sonata Bis Grant 
\[\textbf{2023/50/E/ST1/00025.} \] For the purpose of Open Access, the author has applied a CC-BY public copyright license to any Author Accepted Manuscript (AAM) version arising from this submission.

%\bibliographystyle{abbrv}
%\bibliography{dem}

\medskip 

\noindent
Piotr Pokora, Tomasz Szemberg\\
\noindent
Department of Mathematics,
University of the National Education Commission Krakow,
Podchor\c a\.zych 2,
PL-30-084 Krak\'ow, Poland. \\
Email: \url{piotr.pokora@uken.krakow.pl}, \url{tomasz.szemberg@gmail.com}.
%\begin{thebibliography}{99}
%\end{thebibliography}

\end{document}